\documentclass[11pt,twoside]{amsart}
\usepackage{amsxtra}
\usepackage{color}
\usepackage{amsmath,amsthm,amssymb,bm}
\usepackage{mathrsfs,mathtools}
\usepackage{stmaryrd}
\usepackage{hyperref}
\usepackage{enumerate}
\usepackage{xcolor}
\usepackage{pifont}
\usepackage{adjustbox}
\usepackage{tikz}

\theoremstyle{theorem}
\newtheorem{theorem}{Theorem}[section]
\newtheorem{corollary}[theorem]{Corollary}
\newtheorem{proposition}[theorem]{Proposition}
\newtheorem{lemma}[theorem]{Lemma}
\newtheorem*{mtheorem}{Main Theorem}
\newtheorem*{corollary*}{Corollary}
\newtheorem*{proposition*}{Proposition}

\theoremstyle{definition}
\newtheorem{definition}[theorem]{Definition}

\theoremstyle{remark}
\newtheorem{remark}[theorem]{Remark}

\newcommand{\R}{{\mathbb R}}
\renewcommand{\a}{\alpha}
\renewcommand{\b}{\beta}

\newcommand{\f}{\varphi}
\newcommand{\g}{\gamma}

\newcommand{\cH}{\mathcal{H}}
\newcommand{\SO}{{\mathrm {SO}}}

\newcommand{\G}{{\mathrm G}}
\newcommand{\K}{{\mathrm K}}

\newcommand{\beq}{\begin{equation}}
\newcommand{\eeq}{\end{equation}}
\newcommand{\n}{\nabla}
\newcommand{\W}{\wedge}

\DeclareMathOperator\Ad{Ad}
\DeclareMathOperator\vol{vol}
\newcommand{\Ric}{{\rm Ric}}

\renewcommand{\gg}{\mathfrak{g}}

\newcommand{\gk}{\mathfrak{k}}
\newcommand{\gm}{\mathfrak{m}}

\numberwithin{equation}{section}

\title[3D positively curved generalized Ricci solitons with SO(3) symmetries]{Three-dimensional positively curved generalized Ricci solitons with
SO(3)-symmetries}
\author{Fabio Podest\`a and Alberto Raffero}
\address{Dipartimento di Matematica e Informatica ``U.~Dini'' \\ Universit\`a degli Studi di Firenze\\ Viale Morgagni 67/a\\ 50134 Firenze\\ Italy\\ }
\email{fabio.podesta@unifi.it}
\address{Dipartimento di Matematica ``G.~Peano'' \\ Universit\`a degli Studi di Torino\\
Via Carlo Alberto 10\\
10123 Torino\\ Italy}
\email{alberto.raffero@unito.it}

\begin{document}
\begin{abstract}
We prove the existence of a one-parameter family of pairwise non-isometric, complete, positively curved,
steady generalized Ricci solitons of gradient type on $\R^3$ that are invariant under the natural cohomogeneity one action of SO(3).
In the context of generalized Ricci flow, this result represents the analogue of Bryant's construction
of the complete rotationally invariant steady soliton for the Ricci flow.
\end{abstract}

\maketitle

\section{Introduction}\label{intro}
Let $(M,g)$ be an oriented Riemannian manifold and let $H\in\Omega^3(M)$ be a closed $3$-form on it.
The pair $(g,H)$ is said to be a {\em generalized Ricci soliton} (shortly {\em GRS}) if
there exist a vector field $X\in\Gamma(TM)$ and a real number $\lambda$ such that
\beq\label{GRSIntro}
\begin{split}
\Ric_g 	&=  \lambda\,g -\frac12\mathcal{L}_Xg+\frac14 \cH_{g,H},\\
\Delta_gH	&= 2\lambda\,H-\mathcal{L}_XH.
\end{split}
\eeq
Here, $\Delta_{g} = d\delta_g+\delta_gd$ is the Hodge Laplacian determined by $g$ and the orientation, and
$\cH_{g,H}$ is the symmetric $2$-tensor defined as follows
\[
\cH_{g,H}(X,Y) \coloneqq g(\imath_XH,\imath_YH) = 2 \star_g\left(\imath_XH\wedge \star_g (\imath_YH)\right),
\]
for all $X,Y\in\Gamma(TM)$. According to the sign of the expansion constant $\lambda$, a GRS is said to be
{\em expanding}, {\em steady} or {\em shrinking} if $\lambda<0$, $\lambda=0$ or $\lambda>0$, respectively.
Moreover, it is said to be of {\em gradient type} if $X=\nabla f$, for some $f\in C^\infty(M)$.

\smallskip

Generalized Ricci solitons arise in the study of the {\em generalized Ricci flow} (shortly {\em GRF}), a geometric flow evolving
a family of Riemannian metrics $g_t$ and closed $3$-forms $H_t$ as follows
\begin{equation}\label{GRF}
\left\{
\begin{split}
\tfrac{\partial}{\partial t} g_t &= -2\Ric_{g_t} +\frac12 \cH_{g_t,H_t},\\
\tfrac{\partial}{\partial t} H_t &= -\Delta_{g_t}H_t.
\end{split}
\right.
\end{equation}
In detail, the solution to the generalized Ricci flow starting at a GRS $(g,H)$ is {\em self-similar}, namely it evolves along symmetries of the flow.
In this case, symmetries are diffeomorphisms and simultaneous scalings $(g,H)\to (\a g,\a H)$, for $\a\in\R^+$.
On the other hand, given a self-similar solution $(g_t,H_t)= (\sigma_t\f_t^*g,\sigma_t\f_t^*H)$ of \eqref{GRF}, where
$\sigma_t$ is a positive scalar such that $\sigma_0=1$ and $\f_t\in\mathrm{Dif{}f}(M)$ is a family of diffeomorphisms satisfying $\f_0=\mathrm{Id}_M$,
then the initial datum $(g_0,H_0)$ solves the GRS equation \eqref{GRSIntro} (see Section \ref{GRSSect} for more details).

\smallskip

As in the case of the Ricci flow,  self-similar solutions to the generalized Ricci flow are expected as long time limits and singularity models for the flow.
Indeed, on compact manifolds, the generalized Ricci flow admits a Perelman-type energy monotonicity formula,
and it is the gradient flow of the lowest eigenvalue of a certain Schr\"odinger operator \cite{OSW}.
In this context, critical points of the relevant energy functional are given by steady generalized Ricci solitons of gradient type,
see, e.g., \cite{OSW} or \cite[Sect.~6.1.3]{GFS}.
In particular, every steady GRS on a compact manifold is of gradient type \cite[Cor.~6.11]{GFS}.
Furthermore, it follows from \cite[Prop.~6.4]{Str} (see also \cite[Cor.~6.11]{StTi2}) that an expanding GRS on a compact
manifold is necessarily of the form $(g,H\equiv0)$, with $g$ an Einstein metric.
The dynamical behavior of pairs $(g,H\equiv0)$, with $g$ a Ricci-flat metric, and, more generally,
of gradient steady GRS under the generalized Ricci flow was studied in \cite{RV} and in \cite{L}, respectively

\smallskip

The generalized Ricci flow first appeared in theoretical physics as the renormalization group flow
of a nonlinear sigma model arising in string theory \cite{CFMP}.
From the mathematical viewpoint, it can be considered as a generalization of Hamilton's Ricci flow to metric connections with totally
skew-symmetric torsion \cite{Str},  and as a flow of generalized metrics on exact Courant algebroids \cite{Gar,GFS, Str1}.
In Hermitian non-K\"ahler Geometry, Streets and Tian's {\em pluriclosed flow} \cite{StTi}
can be regarded as a gauge modified version of the generalized Ricci flow, see \cite[Thm.~6.5]{StTi2}.
Examples of solitons of the pluriclosed flow on compact complex $4$-manifolds and related classification results
were recently obtained by Streets and Ustinovskiy in \cite{Str2,StUs}. In \cite{StUs2}, they also provided a partial classification of complete
$4$-dimensional gradient solitons for the generalized K\"ahler-Ricci flow introduced in \cite{StTi1}.
These satisfy equation \eqref{GRSIntro} with $\lambda=0$ and $X$ a gradient vector field.
In the homogeneous setting, examples of left-invariant GRS on the three-dimensional Heisenberg Lie group were given in \cite{Par}.
Furthermore, a study of the homogeneous generalized Ricci flow and its homogeneous solitons is done in \cite{FLS}.

\smallskip

Some known meaningful geometric equations can be seen as special instances of the GRS equation \eqref{GRSIntro}.
For instance, when $\lambda=0$ and the vector field $X$ is either zero or, more generally, $\mathcal{L}_Xg=0$ and $\mathcal{L}_XH=0$,
then \eqref{GRSIntro} reduces to
\[
\Ric_g = \frac14 \cH_{g,H},\qquad
\Delta_gH	= 0.
\]
A pair $(g,H)$ solving the above equations is a fixed point of the GRF, and thus a {\em Bismut Ricci-flat pair} (shortly {\em BRF pair})
in the sense of \cite{PR1}.
In Proposition \ref{GRStoBRF}, we show that a compact steady GRS $(g,H)$ with constant scalar curvature and constant $g$-norm of $H$ must
be a BRF pair. This applies, for instance, to compact homogeneous steady GRS.
The existence of invariant BRF pairs on compact homogeneous spaces was first proved in \cite{PR1}, and further infinite families of examples
were subsequently provided in \cite{PR2} and then in \cite{LW}.

On the other hand, if $H\equiv0$, then \eqref{GRSIntro} reduces to the classical {\em Ricci soliton} equation
\[
\Ric_g =  \lambda\,g -\frac12\mathcal{L}_Xg.
\]

In the unpublished work \cite{Bryant} (see also \cite[Ch.~1, Sect.~4]{CC}),
Bryant proves the existence of an $\SO(3)$-invariant complete steady Ricci soliton of gradient type on $\R^3$,
and he shows that it is positively curved and unique up to homothety.
This motivates the search of complete steady generalized Ricci solitons with rotational symmetry and non-zero $3$-form $H$ on $\R^3$,
see \cite[Question 4.32]{GFS}.
In this article, we show the existence of such solitons, proving the following.
\begin{mtheorem}
There exists a one-parameter family of pairwise non-isometric,
complete, positively curved, $\SO(3)$-invariant steady generalized Ricci solitons of gradient type
$(g_\ell,H_\ell)$ on $\mathbb R^3$, where $\ell\in\mathbb R$.
\end{mtheorem}

We now describe the strategy of the proof, which starts in Section \ref{SecEq}.
We first focus on the open subset $\R^3\smallsetminus\{0\}\simeq \R^+\times S^2$, where the generic $\SO(3)$-invariant metric is a warped
product of the form
\[
g = dt^2+\phi(t)^2 d\sigma^2.
\]
Here, $t$ is the standard coordinate on $\R^+$ and describes the normal geodesic $\gamma_t$,
while $d\sigma^2$ is the $\SO(3)$-invariant metric of constant sectional curvature $1$ on
$S^2\cong\SO(3)/\SO(2)$. We then consider the generic $\SO(3)$-invariant closed $3$-form $H$,
which is determined by a single function $h(t)$,
and the $\SO(3)$-invariant gradient vector field $X=\nabla f(t) = f'(t)\frac{\partial}{\partial t}$.
It is known that the metric $g$ extends to a smooth metric on the whole $\R^3$
if and only if the function $\phi$ extends as an odd smooth function about the origin with $\phi'(0)=1$.
Moreover, the gradient vector field $X$ extends to the whole $\R^3$ if and only if $f(t)$ extends as an even smooth function about the origin,
and it is not restrictive assuming the normalization condition $f(0)=0$.

Plugging the expressions of $g$, $H$ and $X$ into the system \eqref{GRSIntro} with $\lambda=0$,
we obtain that $h(t) = k\,\phi^2\,e^f$, for some constant $k$, and that the pair $(\phi,f)$ must solve the
following singular system of second order differential equations
\beq\label{SysIntro}
\left\{\begin{split}
\phi'' 	&=  \frac {1-\phi'^2}{\phi}+\phi'f'- \frac{k^2}{2}\phi\,e^{2f},\\
f'' 	&=2\frac {1-\phi'^2}{\phi^2}+2\frac{\phi'}{\phi}f' -  \frac{k^2}{2} e^{2f}.
\end{split}\right.
\eeq
The system for $\SO(3)$-invariant steady Ricci solitons is obtained putting $k=0$.
Since we are interested in solutions with non-zero $H$,
we focus on the case $k\neq0$, where it is not restrictive assuming that $k^2=2$, up to homothety (see Remark \ref{Remc1} for details).

\smallskip

In Section \ref{SecCL}, we obtain a conservation law for the system \eqref{SysIntro} from the following conservation law for
steady gradient generalized Ricci solitons 
\[
 d\left(\Delta f +\frac{1}{6}|H|_g^2 - |\nabla f|_g^2\right) = 0.
\]

This will be a fundamental tool in proving the existence of complete solitons.
\begin{proposition}\label{IntroCL}
The system \eqref{SysIntro} has the following conservation law
\[
0 = d\left(\Delta f +\frac{1}{6}|H|_g^2 - |\nabla f|_g^2\right) = \left(f'' +2\frac{\phi'}{\phi}f' +k^2e^{2f} - (f')^2\right)'.
\]
\end{proposition} \noindent
In the same section, we also show that this conservation law reduces to the one obtained by Bryant \cite[eq.~(2.5)]{Bryant} when $k=0$.

In Section \ref{SecSTE}, we prove the existence of a solution to \eqref{SysIntro} defined for short times and satisfying the required
extendability conditions.
\begin{proposition}\label{IntroSTE}
For every $q\in \mathbb R$, there exist $\epsilon>0$ and a smooth map $(\phi,f):(-\epsilon,\epsilon)\to \R^2$
which gives a solution to the system \eqref{SysIntro} with $k^2=2$
on $(-\epsilon,\epsilon)\smallsetminus\{0\}$ and satisfies the following conditions
\begin{enumerate}[i)]
\item $\phi$ is odd with $\phi'(0)=1$;
\item $f$ is even with $f(0)=0$;
\item $\phi^{(3)}(0) = \tfrac12\left( f''(0)-1\right) =q$.
\end{enumerate}
\end{proposition}
This result is achieved using Malgrange Theorem \cite[Thm.~7.1]{M}.
The discussion highlights the freedom in the choice of the real constant $q=\phi^{(3)}(0)$. This ensures the existence of a one-parameter
family of local solutions to \eqref{SysIntro}, and thus the existence of a one-parameter
family of $\SO(3)$-invariant steady generalized Ricci solitons defined on small balls in $\R^3$.

\smallskip

In Section \ref{SecComplete}, we prove that the local solutions extend to the whole $[0,+\infty)$
so that the corresponding $\SO(3)$-invariant steady generalized Ricci solitons are complete.
The proof follows from a thorough qualitative analysis of the solutions to the system \eqref{SysIntro}.
The crucial result is an estimate for the function $\phi\,e^f$, which is proved in Proposition \ref{PropEst}.
We summarize our results in the next theorem, whose proof is obtained in Section \ref{SecComplete} through several steps.
\begin{theorem}\label{IntroCOMP}
Let $q_\ell \coloneqq -\frac{35}{12}-e^{-\ell}$, where $\ell\in\R$.
Then, for every $\ell\in\R$, the local solution $(\phi_\ell,f_\ell)$ to \eqref{SysIntro} with $k^2=2$ and
$\phi_\ell^{(3)}(0) = \tfrac12\left( f_\ell''(0)-1\right)  =q_\ell$
extends to the whole $[0,+\infty)$ giving rise to a complete steady generalized Ricci soliton $(g_\ell,H_\ell)$
with vector field $X_\ell=\nabla f_\ell$ on $\mathbb R^3$. Moreover,
\begin{enumerate}[i)]
\item\label{COMPi}
on $(0,+\infty)$ the function $\phi_\ell$ is positive and concave, its derivative satisfies $0<\phi_\ell'<1$, and
\[
\phi_\ell \sim \sqrt{\frac{2t}{|6q_\ell+5|}},\quad t\to+\infty;
\]
\item\label{COMPii}
on $(0,+\infty)$ the function $f_\ell$ is negative, strictly decreasing  and
\[
f_\ell \sim -\sqrt{|6q_\ell+5|} \,t,\quad t\to+\infty;
\]
\item\label{COMPiii}
$g_\ell$ has positive curvature and
\[
q_\ell = -\lim_{t\to 0}K_{\mathrm{tan}}(t),
\]
where $K_{\mathrm{tan}}(t)$ denotes the sectional curvature of the orbit $\SO(3)\cdot\gamma_t$;
\item\label{COMPiv}
$g_\ell$ and $g_{\ell'}$ are isometric if and only if $q_\ell = q_{\ell'}$;
\item\label{COMPv}
the $3$-form $H_\ell$ is given by
\[
H_\ell = \frac{\sqrt 2}{t^2}\phi_\ell^2 e^{f_\ell} \vol_o,
\]
where $\vol_o$ denotes the euclidean volume form, and it converges to zero  exponentially fast for $t\to+\infty$.
\end{enumerate}
\end{theorem}

\medskip
\noindent
{\bf Notation.} Throughout the paper, Lie groups will be denoted by capital letters and their  Lie algebras will be denoted by the respective gothic letters.
When a Lie group $\G$ acts on a manifold $M$, the vector field associated to any $X\in\gg$ will be denoted by $\widehat X$.

\section{Generalized Ricci solitons}\label{GRSSect}

In this section, we review the correspondence between generalized Ricci solitons and self-similar solutions of the generalized Ricci flow.
We begin recalling the following.

\begin{definition}
A Riemannian metric $g$ and a closed $3$-form $H$ on $M$ define a {\em generalized Ricci soliton} (shortly {\em GRS}) if
there exist a vector field $X\in\Gamma(TM)$ and a real number $\lambda$ such that
\begin{equation}\label{GRSeq}
\begin{split}
\Ric_g 	&=  \lambda\,g -\frac12\mathcal{L}_Xg+\frac14 \cH_{g,H},\\
\Delta_gH	&= 2\lambda\,H-\mathcal{L}_XH.
\end{split}
\end{equation}
A GRS is said to be {\em expanding}, {\em shrinking} or {\em steady} if $\lambda<0$, $\lambda>0$ or $\lambda=0$, respectively.
A GRS is said to be {\em of gradient type} if $X= \nabla f$, for some $f\in C^\infty(M)$. In such a case, the first equation in \eqref{GRSeq} becomes
\beq\label{GRSgrad}
\Ric_g =  \lambda\,g -\mathrm{Hess}_g(f)+\frac14 \cH_{g,H}.
\eeq

\end{definition}

\begin{remark}
If the pair $(g,H)$ defines a GRS with corresponding vector field $X$ and expansion constant $\lambda$,
then for every $\a\in \mathbb R^+$
the pair $(\a g, \a H)$ also defines a GRS with corresponding vector field $\a^{-1}X$ and expansion constant $\a^{-1} \lambda$.
Moreover also $(g,-H)$ defines a GRS.
\end{remark}

\smallskip

The right-hand side of the generalized Ricci flow equation \eqref{GRF} is invariant under diffeomorphisms and under the
scaling $(g,H)\mapsto(\alpha\,g,\alpha\,H)$, where $\a\in\R^+$.
Indeed, given any $\a\in\R^+$, it is well-known that $\Ric_{\a g} = \Ric_g$,
and a straightforward computation shows that
\[
\begin{split}
\cH_{\a g,\a H}(X,Y) &= (\a g)(\imath_X\a H,\imath_Y \a H) = \frac{1}{\a^2}g(\imath_X\a H,\imath_Y \a H) = \cH_{g,H}(X,Y),\\
\Delta_{\a g}(\a H) &= \frac{1}{\a}\Delta_g(\a H) = \Delta_gH.
\end{split}
\]
On the other hand, given any diffeomorphism $\varphi \in \mathrm{Dif{}f}(M)$, one has $\Ric_{\f^*g} = \f^*\Ric_g$. Moreover,
the identity $\star_{\f^*g} = \f^*\star_g(\f^{-1})^*$ implies that $\delta_{\f^*g} = \f^*\delta_g(\f^{-1})^*$, whence the following identities follow
\[
\begin{split}
\cH_{\f^*g,\f^*H}(X,Y)	&=	2 \star_{\f^*g}\left(\imath_X \f^*H \wedge \star_{\f^*g} (\imath_Y \f^*H)\right)\\
					&=	2 \f^*\left[ \star_{g}\left(\imath_{\f_*X} H\wedge \star_{g} (\imath_{\f_*Y}  H)\right)\right] = (\f^*\cH_{g,H})(X,Y),\\
\Delta_{\f^*g}(\f^*H)		&=	\f^*\Delta_g(\f^{-1})^*(\f^*H) = \f^*\Delta_gH.
\end{split}
\]

\smallskip

The previous properties allow one to characterize self-similar solutions of the GRF starting at a given pair $(g(0),H(0)) = (g,H)$ by means
of the GRS equation \eqref{GRSeq} for $(g,H)$.

Assume that $(g_t,H_t) \coloneqq (\sigma_t\f_t^*g,\sigma_t\f_t^*H)$ is a {\em self-similar} solution to the GRF \eqref{GRF}
starting at $(g,H)$,where $\sigma_t$ is a positive scalar satisfying $\sigma_0=1$, and $\f_t\in\mathrm{Dif{}f}(M)$ is a one-parameter family
of diffeomorphisms such that $\f_0=\mathrm{Id}_M$. Then, differentiating with respect to $t$ and evaluating at $t=0$, one gets
\[
-2\Ric_{g} +\frac12 \cH_{g,H} 	= \left.\frac{\partial}{\partial t}g_t \right|_{t=0}
						= \left.\left(\sigma'_t\f_t^*g+\sigma_t\f_t^*(\mathcal{L}_Xg)\right)\right|_{t=0}
						= \sigma'_0\,g + \mathcal{L}_Xg,
\]
where $X\in\Gamma(TM)$ is the vector field such that $X_{\f_t(x)} = \frac{d}{dt}\f_t(x)$, for all $x\in M.$
Thus, $g$ solves the first equation in \eqref{GRSeq} with $\lambda = -\tfrac12 \sigma'_0$. Similarly, one has
\[
-\Delta_gH	= \left.\frac{\partial}{\partial t}H_t \right|_{t=0}
			= \left.\left(\sigma'_t\f_t^*H+\sigma_t\f_t^*(\mathcal{L}_XH)\right)\right|_{t=0}
			= \sigma'_0\,H + \mathcal{L}_XH,
\]
and thus $H$ solves the second equation in  \eqref{GRSeq}.

Conversely, assume that the pair $(g,H)$ is a GRS solving \eqref{GRSeq}.
Let $\sigma_t\coloneqq 1- 2\lambda t$ be defined on the maximal interval $I\subseteq \R$ where it is positive,
and let $\f_t\in\mathrm{Dif{}f}(M)$ be the family of diffeomorphisms generated by the vector field $Y_t \coloneqq \tfrac{1}{\sigma_t} X$
and satisfying $\f_0=\mathrm{Id}_M.$
Then, the pair $(g_t,H_t) = (\sigma_t\f_t^*g,\sigma_t\f_t^*H)$ is a self-similar solution to the GRF \eqref{GRF}. Indeed, one has
\[
\begin{split}
\frac{\partial}{\partial t}g_t 	&= \sigma'_t\f_t^*g+\sigma_t\f_t^*(\mathcal{L}_{Y_t} g)	
					= \f_t^*\left(-2\lambda g + \mathcal{L}_X g \right)
					= \f_t^*\left(-2\Ric_g +\frac12 \cH_{g,H}\right)\\
					&= -2\Ric_{\f_t^*g} +\frac12 \cH_{\f_t^*g,\f_t^*H}
					= -2\Ric_{g_t} +\frac12 \cH_{g_t,H_t},
\end{split}
\]
and, similarly,
\[
\begin{split}
\frac{\partial}{\partial t}H_t 	&= \sigma'_t\f_t^*H+\sigma_t\f_t^*(\mathcal{L}_{Y_t} H)
					= \f_t^*\left(-2\lambda H + \mathcal{L}_X H \right)
					= \f_t^*\left(-\Delta_gH\right)\\
					&= -\Delta_{\f_t^*g}(\f_t^*H)
					= -\Delta_{g_t}H_t.
\end{split}
\]

It is well-known that compact Ricci solitons with constant scalar curvature must be Einstein.
We conclude this section with an analogous result for steady generalized Ricci solitons.
\begin{proposition}\label{GRStoBRF}
Let $(g,H)$ be a steady generalized Ricci soliton on a compact manifold $M,$ and assume that
$R_g - \tfrac14 |H|_g^2$ is constant.
Then, $(g,H)$ is a BRF pair, namely a fixed point of the generalized Ricci flow.
\end{proposition}
\begin{proof}
By \cite[Cor.~6.11]{GFS}, a compact steady GRS is of gradient type, so that $X=\nabla f$, for some $f\in C^\infty(M)$.
The first equation in \eqref{GRSeq} becomes then \eqref{GRSgrad}. Tracing it with respect to $g$ and recalling that $\lambda=0$, we get
\[
R_g = -\Delta f+\frac14|H|_g^2.
\]
Integrating over $M,$ we see that that $R_g - \frac14 |H|_g^2=0$. This implies that $f$ is constant, and our claim follows.
\end{proof}

\begin{corollary}
A compact homogeneous steady generalized Ricci soliton is a BRF pair.
\end{corollary}

\begin{remark}
The thesis of Proposition \ref{GRStoBRF} also holds when $(g,H)$ is a GRS of gradient type and
$R_g - \tfrac14 |H|_g^2$ is constant.
\end{remark}

\section{Rotationally invariant Generalized Ricci Solitons on $\R^3$: the equations}\label{SecEq}
Consider the group $\G\coloneqq\SO(3)$ acting in a standard way on $\mathbb R^3$ with principal isotropy $\K\cong\SO(2)$,
where $\gk$ is generated by the matrix $A\coloneqq E_{32}-E_{23}$. Then, $\gg=\gk+\gm$, where the $\Ad(\K)$-stable subspace $\gm$ is generated by
\[
e_1\coloneqq E_{12}-E_{21},\quad e_2 \coloneqq E_{13}-E_{31}.
\]
We have $[A,e_1]=e_2$, $[A,e_2]=-e_1$ and $[e_1,e_2]=A$.\par

\smallskip

We consider the metric $g$ defined on $\mathbb R^3\smallsetminus\{0\} \simeq \R^+\times S^2$ as
\beq\label{g}
g= dt^2+\phi^2 d\sigma^2,
\eeq
where $d\sigma^2 = e^1\odot e^1 + e^2\odot e^2$ is the metric of constant sectional curvature $1$ on $S^2\cong \G/\K$,
and $\phi$ is a positive smooth function.
It is well-known (see, e.g., \cite[Lemma A.2]{CC}) that the metric $g$ defined in \eqref{g} extends to a smooth metric on the whole $\mathbb R^3$
if and only if the function $\phi$ extends smoothly as an {\it odd} smooth function on $\R$ with $\phi'(0)=1$.

We then consider a $\G$-invariant $3$-form
\beq\label{H}
H = h(t)\, dt\wedge e^{12},
\eeq

where $e^{12}$ is the volume form of the metric $d\sigma^2$ on $S^2\cong \G/\K$. Notice that $H$ is closed as
\[
dH = h'(t)\,dt\wedge dt\wedge e^{12} - h(t)\, dt\wedge de^{12}=0,
\]
and that it admits the potential $\left(\int_0^th(s)ds\right) e^{12}$.
We now compute
\[
\star_g H = \frac{h}{\phi^2},\qquad d\star_g H= \left(\frac{h}{\phi^2}\right)'dt,\qquad
\delta_gH=-\star_g d\star_g H = -\phi^2\left(\frac{h}{\phi^2}\right)' e^{12}.
\]
Note that
\beq\label{EqNormH}
|H|_g^2 = H_{ijk}g^{il}g^{jm}g^{kn}H_{lmn} = 6\star_g(H \W \star_g H) =  6\frac{h^2}{\phi^4}.
\eeq
Moreover, the symmetric $2$-tensor $\cH_{g,H}$ has the following expression
\beq\label{H2}
\cH_{g,H} = 2\,\frac{h^2}{\phi^4}\, dt^2 + 2\,\frac{h^2}{\phi^2} d\sigma^2 = 2\frac{h^2}{\phi^4}g.
\eeq

Let  $(t,y,z)$ denote the standard coordinates on $\mathbb R^3$, and
fix the normal geodesic $\gamma_t=(t,0,0)$, so that $\hat e_1|_{\g_t}= -t\frac{\partial}{\partial y}$ and
$\hat e_2|_{\g_t}= -t\frac{\partial}{\partial z}$.
As both $H$ and the standard volume form $\vol_o$ are $\G$-invariant, we see that $H= \frac{h}{t^2}\,\vol_o$ along $\g$.
Using this, we see that $H$ admits a smooth extension to the whole $\mathbb R^3$
if and only if $h$ extends as an even function with $h(0)=0$.\par

\smallskip

We now consider the $\G$-invariant vector field $X=\nabla f= f'\,\xi$, where $\xi\coloneqq \frac{\partial}{\partial t}$,
for some smooth even function $f$.
Recall that for every pair of vector fields $Y,Z$ the following identities hold $\mathcal L_Xg(Y,Z) = 2 g(\nabla_YX,Z) = 2YZf - 2\nabla_YZ\cdot f$.
We then obtain
\[
\begin{split}
\mathcal L_Xg(\xi,\xi)		&= 2g(\n_\xi f'\xi,\xi) = 2f'',\\
\mathcal L_Xg(e_1,e_1)	&= X\cdot g(e_1,e_1) = f'g(e_1,e_1)' = 2\phi\phi' f'.
\end{split}
\]
Moreover
\[
\mathcal L_XH(\xi,e_1,e_2) = X \left(H(\xi,e_1,e_2)\right) - H([X,\xi],e_1,e_2) = f'h' + f'' h = (f'h)'.
\]

Using the expression of the Ricci tensor of the metric $g$ (see, e.g., \cite[Ch.~7, Cor.~43]{Oneill}),
the equations \eqref{GRSeq} can be written as follows
\begin{equation}\label{GRSeq1}
\left\{\begin{split}
1-(\phi')^2 -\phi\phi'' 	&=  \lambda\,\phi^2-\phi\phi'\, f' + \frac {h^2}{2\phi^2},\\
-2\phi\phi'' &=(\lambda - f'')\phi^2 + \frac {h^2}{2\phi^2},\\
\left(\phi^2\left(\frac{h}{\phi^2}\right)'\right)' &= -2\lambda\,h+ (f'h)',
\end{split}\right.
\end{equation}
where the first equation is $\Ric(e_1,e_1)$ and the second one is $\phi^2 \Ric(\xi,\xi)$.

When $\lambda =0$, the system \eqref{GRSeq1} becomes
\begin{equation}\label{GRSeq2}
\left\{\begin{split}
1-(\phi')^2 -\phi\phi'' 	&=  -\phi\phi'\, f' + \frac {h^2}{2\phi^2},\\
-2\phi\phi'' &=- f''\phi^2 + \frac {h^2}{2\phi^2},\\
\left(\phi^2\left(\frac{h}{\phi^2}\right)'\right)' &=  (f'h)'.
\end{split}\right.
\end{equation}
From the last equation we obtain
\beq\label{cost}
\phi^2\left(\frac{h}{\phi^2}\right)' - f'h = \mbox{\rm{const}}.
\eeq
We note that the function $\frac{h}{\phi^2}$ extends smoothly over the origin, so that the left-hand side of \eqref{cost} vanishes at the origin,
forcing the constant on the right-hand side to vanish. Therefore
\[
\phi^2\left(\frac{h}{\phi^2}\right)' = f'h.
\]
If we set $w\coloneqq \frac{h}{\phi^2}$, then $w$ must extend over the origin as an even function and it satisfies for $t>0$
\[
w'=f'w.
\]
Hence, there exists $k\in\mathbb R$ so that $w= k\, e^f$, namely
\beq\label{h}
h=k\,  \phi^2e^f.
\eeq
Note that $f$ is determined up to an additive constant, so we can suppose to have $f(0)=0$.
The system  \eqref{GRSeq2} is now written as
\begin{equation}\label{GRSeq33}
\left\{\begin{split}
1-(\phi')^2 -\phi\phi'' 	&=  -\phi\phi'\, f' + c^2 e^{2f}\phi^2,\\
-2\phi\phi'' &=- f''\phi^2 + c^2 e^{2f}\phi^2,\\
\end{split}\right.
\end{equation}
where $c^2=\frac{k^2}2\geq 0$, the function  $\phi$ is odd with $\phi'(0)=1$ and $f$ is even with $f(0)=0$.
The parameter $k$ can be taken to be non-negative, as $(g,H)$ being a GRS implies $(g,-H)$ is so as well.
The equation for steady Ricci solitons is obtained when $k=0$.
It is proved in \cite{Bryant} that there exists a unique complete $\G$-invariant steady Ricci soliton on $\R^3$, up to homothety.

\begin{remark}
Special solutions of the above system are given by rotationally invariant BRF pairs.
It is well-known that such solutions are exhausted by constant curvature metrics together with $3$-forms given by multiples of their volume forms,
see \cite[Prop.~3.55]{GFS}.
Explicitly, if we put $f=0$, the system \eqref{GRSeq33} gives
\[
\left\{\begin{split}
1-(\phi')^2 -\phi\phi'' 	&=   c^2 \phi^2,\\
\phi'' &= -\frac {c^2}2 \phi,\\
\end{split}\right.
\]
whose solution satisfying the boundary condition $\phi(0)=0$, $\phi'(0)=1$ is $\phi(t) = \frac 1\b \sin(\b t)$, where $\b\coloneqq\sqrt{\frac{c^2}2}\neq 0$.
This is defined on $\left[0,\frac{\pi}{\b}\right)$ and converges to the flat solution when $\b\to 0$.
\end{remark}

\begin{remark}\label{Remc1}
Given a GRS ($g,H,X)$ and $\b\in\mathbb R$, $\b\neq 0$,
we consider the data $(\tilde g \coloneqq \b^2\, g,\tilde{H}\coloneqq \b^2H,\tilde{X}\coloneqq \b^{-2}X)$.
Putting $s=\b\, t$, we have
\[
\tilde g = \b^2 g = d(\b t)^2 + \b^2 \phi^2 d\sigma^2 = ds^2 + \left(\b \phi\left(\frac s\b\right)\right)^2d\sigma^2 =  ds^2 + \left(\tilde \phi(s)\right)^2d\sigma^2,
\]
with $\tilde\phi(s) = \b \phi(\frac s\b)$. Moreover
\[
\frac 1{\b^2} \nabla f= \frac 1{\b^2} X = \tilde X = \widetilde{\nabla} \tilde f = \frac 1{\b^2} \nabla \tilde f,
\]
so that $\tilde f(s) = f(\frac s\b)$, as both vanish at $t=0$, $s=0$.
We then have $\tilde H = \b^2 H $, hence
\[
\tilde h(s)\, ds\wedge e^{12} = \tilde H = \b^2\, H = \b^2\, h(t)\, dt\wedge e^{12} = \b\,h\left(\frac s\b\right)\, ds\wedge e^{12},
\]
so that
\[
\tilde h(s) = \b\, h\left(\frac s\b\right).
\]

When $H\equiv0$, i.e., when we consider the steady Ricci soliton equation, the system \eqref{GRSeq33} is invariant under the transformation
$(\phi,f)\mapsto (\tilde \phi,\tilde f)$, while in case $H\neq 0$ this invariance breaks down.
If we now compare
\[
h=k\, e^f \phi^2,\qquad \tilde h=\tilde k\, e^{\tilde f} \tilde \phi^2,
\]
we get
\[
k = \tilde k\, \b.
\]
This means that, up to homothety, we can suppose $c=\frac{k^2}{2}=1$  whenever the $3$-form $H$ is not zero.
\end{remark}

\begin{remark}
The steady generalized Ricci soliton equation implies the existence of a closed $b$-field symmetry, namely
$b = \delta_gH +\imath_XH$ (cf.~\cite[Sect.~4.4.2]{GFS}). The above computations show that $b=0$ in the radially symmetric Ansatz we are considering:
\[
b = \delta_gH +\imath_XH = \left( -\phi^2\left(\frac{h}{\phi^2}\right)' +f'h \right)e^{12} = 0.
\]

\end{remark}

\section{A conservation law}\label{SecCL}
In this section we prove Proposition \ref{IntroCL}, establishing the existence of a conservation law for the system \eqref{GRSeq33}.
This will turn out to be a fundamental tool in proving the existence of complete solitons.

\begin{proof}[Proof of Proposition \ref{IntroCL}]
Recall the following formula for steady generalized Ricci solitons of gradient type \cite[Prop.~4.33, (3)]{GFS}
\beq\label{NablaSteady}
d\left(R_g -\frac{1}{12}|H|_g^2 + 2\Delta f - |\nabla f|_g^2\right) = 0,
\eeq
where $f$ is the potential function of the gradient vector field $X$.

Tracing the gradient soliton equation \eqref{GRSgrad} with respect to $g$ and putting $\lambda=0$, we obtain
\[
R_g = -\Delta f + \frac14 |H|_g^2.
\]
Moreover, we have
\[
\Delta f 	= \mathrm{div}_g\left(f'\frac{\partial}{\partial t}\right) = \frac{1}{\phi^2}\frac{\partial}{\partial t}\left(\phi^2 f'\right)
		= f'' +2\frac{\phi'}{\phi}f'.
\]
Plugging these expressions into \eqref{NablaSteady} and using \eqref{EqNormH} and \eqref{h}, we get
\beq\label{pint1}
0 	= d\left(\Delta f +\frac{1}{6}|H|_g^2 - |\nabla f|_g^2\right)
	= \left(f'' +2\frac{\phi'}{\phi}f' +k^2e^{2f} - (f')^2\right)'.
\eeq
\end{proof}

We note that the conservation law \eqref{pint1} obtained here coincides with the conservation law in Bryant's paper \cite[eq.~(2.5)]{Bryant}
when $k=0$. Indeed, using the equations \eqref{GRSeq33} and substituting $f''$ and $\phi''$ in \eqref{pint1}, we obtain
\beq\label{pint2}
\left(2\frac{1-(\phi')^2}{\phi^2} + 4\frac{\phi'}{\phi}f' + \frac{k^2}2 e^{2f} - (f')^2\right)' = 0,
\eeq
 and thus
\beq\label{pint}
\left(f'-2\frac{\phi'}{\phi}\right)^2 - 2\frac{1+(\phi')^2}{\phi^2} -\frac{k^2}{2} e^{2f} = C,
\eeq
for some constant $C$.
Throughout the following, we will always take $k^2=2$, i.e., $c=1$ (cf.~Remark \ref{Remc1}).

\section{Short time existence}\label{SecSTE}
In this section we prove Proposition \ref{IntroSTE}.
We begin rewriting the system \eqref{GRSeq33} as follows
\begin{equation}\label{GRSeq3}
\left\{\begin{split}
\phi'' 	&=  \frac {1-\phi'^2}{\phi}+\phi'f'- \phi\,e^{2f},\\
f'' 	&=2\frac {1-\phi'^2}{\phi^2}+2\frac{\phi'}{\phi}f' - e^{2f},
\end{split}\right.
\end{equation}
where we have put $c= \frac{k^2}{2}=1$.
We recall that we are interested in solutions $(\phi,f)$ where $\phi$ is odd with $\phi'(0)=1$ and $f$ is even with $f(0)=0$. Notice also that
\beq\label{linkeqs}
\phi\,f'' -2\phi'' = \phi\,e^{2f}.
\eeq

From \eqref{GRSeq3} we compute $f''(0)$ and $\phi^{(3)}(0)$ using limits and de l'H\^opital rule
\[
\begin{split}
\phi^{(3)}(0)	&= \lim_{t\to 0}\left(-2\frac{\phi'\phi''}{\phi} -\frac 1{\phi^2}\phi'(1-\phi'^2) + \phi''f'+\phi'f'' - e^{2f}\phi'\right)\\
			&= f''(0) -1 +  \lim_{t\to 0}\left(-2\frac{\phi'\phi''}{\phi} -\frac 1{\phi^2}\phi'(1-\phi'^2)\right)\\
			&= f''(0) -1 - 2 \phi^{(3)}(0) - \lim_{t\to 0}\frac 1{\phi^2}(1-\phi'^2)\\
			&= f''(0) -1 - 2 \phi^{(3)}(0) + \phi^{(3)}(0)\\
			&= f''(0) -1 -  \phi^{(3)}(0),
\end{split}
\]
whence
\beq\label{der} f''(0)= 2\phi^{(3)}(0) + 1.\eeq
If we take the limit in the second equation in \eqref{GRSeq3}, we obtain again the same relation \eqref{der}.
This implies that the opening terms of the Taylor expansion of $\phi$ and $f$ are as follows
\[
\begin{split}
\phi(t) 	&= t + \frac 16 qt^3 + o(t^4),\\
f(t) 		&= \frac 12 (2q+c^2)t^2 + o(t^3),
\end{split}
\]
where $q\coloneqq \phi^{(3)}(0)$.

\begin{remark}
Recall from Section \ref{SecEq} that $\mathcal L_Xg(\xi,\xi)= 2f''(t)$. Therefore, the quantity $f''(0)= 2\phi^{(3)}(0) + 1 = 2q+1$ can be interpreted
as (half of) the push given by $X=\nabla f$ to the metric $g$ in the radial direction at the origin.
\end{remark}

We now write ${\phi(t)=t\, a(t)}$, where $a$ is an {\it even} function with $a(0)=1$. We have
\[
\phi'=a+ta',~ \phi''=2a'+ta'',~ 1-(\phi')^2 = 1-a^2-2taa' -t^2a'^2,
\]
hence
\[
\frac{1-(\phi')^2}{\phi} = \frac 1t\,\frac{1-a^2}{a} -t\, \frac{a'^2}{a}-2a',\quad
\frac{1-(\phi')^2}{\phi^2} = \frac 1{t^2}\,\frac{1-a^2}{a^2}-\frac 2t \frac{a'}{a} -\, \frac{a'^2}{a^2},
\]
so that the system \eqref{GRSeq3} reads
\begin{equation}\label{GRSeq4}
\left\{\begin{split}
a'' 	&=  \frac 1{t^2}\left(\frac{1-a^2}{a}\right) +\frac 1t (af'-4a') + \left(a'f'-\frac{a'^2}{a} -a e^{2f}\right), \\
f'' &=\frac 2{t^2}\,\left(\frac{1-a^2}{a^2}\right)+\frac 2t\left( f' -2\frac{a'}{a}\right)+\left(2\frac{a'}af' -\, 2\frac{a'^2}{a^2}- e^{2f}\right).
\end{split}\right.
\end{equation}
The initial conditions are
\[
(a(0),f(0))=(1,0),\quad (a'(0),f'(0))=(0,0).
\]
If we set $P\coloneqq(a,f)$ and $Q\coloneqq(a',f')$, then the singular system \eqref{GRSeq4} can be written as
\begin{equation}\label{GRSeq5}
\left\{\begin{split}
P' 	&= Q, \\
Q' &= \frac 1{t^2} A(P) + \frac 1t B(P,Q) + C(P,Q),
\end{split}\right.
\end{equation}
with initial conditions $P(0)=(1,0),Q(0)=(0,0)$, where
\[
\begin{split}
A((a,f)) 	&= \left(\frac{1-a^2}{a},2\frac{1-a^2}{a^2}\right),\\
B((a,f,a',f'))	&= \left(af'-4a',2\left(f'-2\frac{a'}{a}\right)\right).
\end{split}
\]

Following the approach used in \cite{EW} and based on Malgrange Theorem (see \cite[Thm.~7.1]{M}), we aim at finding formal power series
$\hat P(t)=  \sum_{n=0}^\infty\frac{p_{2n}}{(2n)!}t^{2n}$ and $\hat Q(t) = \sum_{n=1}^\infty\frac{p_{2n}}{(2n-1)!}t^{2n-1}$
which satisfy the system \eqref{GRSeq5}.
As $A(P(0))=0$, we can write
\[
A(\hat P(t))= \sum_{n=1}^{\infty}\frac{a_{2n}}{(2n)!}t^{2n},~
B(\hat P(t),\hat Q(t)) = \sum_{n=0}^\infty \frac{b_{2n+1}}{(2n+1)!}t^{2n+1},~
C(\hat P(t),\hat Q(t)) = \sum_{n=0}^\infty \frac{c_{2n}}{(2n)!}t^{2n},
\]
for some coefficients $a_{2n},b_{2n+1},c_{2n}\in\mathbb R^2$, so that $(\hat P,\hat Q)$ satisfies the system if and only if
\beq\label{p2n}
p_{2n+2}= \frac{a_{2n+2}}{(2n+2)(2n+1)} + \frac{b_{2n+1}}{2n+1} + c_{2n},\quad n\geq 0.
\eeq
Since
\[
\begin{split}
a_{2n+2} 	&= \left.dA\right|_{P(0)}\cdot p_{2n+2} ~ \operatorname{mod}\ (p_0,p_2,\ldots,p_{2n}),\\
b_{2n+1} 	&= 	\left.\frac{d^{2n+1}}{dt^{2n+1}}B( \hat P(t),\hat Q(t))\right|_{t=0}
	= \left.\frac{d^{2n}}{dt^{2n}}\left(\frac{\partial B}{\partial P}\cdot \hat{P}'(t) + \frac{\partial B}{\partial  Q}\cdot \hat{P}''(t)\right)\right|_{t=0} \\
		&= \left.\frac{\partial B}{\partial  Q}\right|_{( P(0),Q(0))}\cdot  p_{2n+2} ~ \operatorname{mod}\ (p_0,p_2,\ldots,p_{2n}),
\end{split}
\]
equation \eqref{p2n} can be rewritten in the form
\[
p_{2n+2} = \frac{1}{(2n+2)(2n+1)} dA|_{P(0)}\cdot p_{2n+2} + \left.\frac{1}{2n+1}\frac{\partial B}{\partial  Q}\right|_{( P(0),Q(0))}\cdot  p_{2n+2} + d_{2n},
\]
where $d_{2n}\in\mathbb R^2$ are suitable expressions in terms of $p_0,p_2,\ldots,p_{2n}$. It follows that
there exists a formal series solution if and only if, for every $n\geq0$,
\[
L_{2n}\cdot   p_{2n+2} = d_{2n},
\]
where $L_{2n}$ is the $2\times 2$-matrix
\[
L_{2n}\coloneqq I - \frac 1{2(n+1)(2n+1)} dA|_{P(0)} - \frac 1{2n+1}\left.\frac{\partial B}{\partial Q}\right|_{(P(0),Q(0))}.
\]

Now
\[
dA|_{P(0)} = \left(\begin{matrix} -2&0\\ -4&0\end{matrix}\right),\qquad
\left.\frac{\partial B}{\partial Q}\right|_{(P(0),Q(0))} = \left(\begin{matrix} -4&1\\ -4&2\end{matrix}\right),
\]
and therefore
\[
L_{2n} =
\left(\begin{matrix} 1+\frac 1{(n+1)(2n+1)}+\frac 4{2n+1} & -\frac 1{2n+1}\\
\frac 2{(n+1)(2n+1)}+\frac 4{2n+1} & 1-\frac 2{2n+1}\end{matrix}\right).
\]
If $n\geq 1$, we have $\det L_{2n} >0$, showing that we can determine $p_{2n+2}$ from $p_0,\ldots,p_{2n}$.
If $n=0$,
\[
\det L_0 = \det\left(\begin{matrix} 6&-1\\6&-1\end{matrix}\right) = 0,
\]
but nevertheless we have already shown that it is possible to determine $a''(0)$ and $f''(0)$ from the system \eqref{GRSeq3},
see \eqref{der}, so that the real parameter $q\coloneqq\phi^{(3)}(0)=3a''(0)$ can be arbitrarily chosen.
This means that for every $q\in\mathbb R$ we can find a formal power series which satisfies the equations of the system \eqref{GRSeq3},
so that by Malgrange Theorem \cite[Thm.~7.1]{M} there exists a smooth solution to the system having that formal power series
as the Taylor expansion at $t=0$.

We now prove Proposition \ref{IntroSTE} on the short time existence of solutions to the system \eqref{GRSeq3} satisfying the required conditions.
This ensures the existence of $\SO(3)$-invariant generalized Ricci solitons on small balls in $\mathbb R^3$.
\begin{proof}[Proof of Proposition \ref{IntroSTE}]
We have already proven the existence of a smooth solution $(\phi,f)$ defined on a neighborhood of $t=0$ and satisfying the conditions
\[
\phi(0)=0,~\phi'(0)=1,~\phi''(0)=0,~\phi^{(3)}(0)=q,\quad f(0)=f'(0)=0.
\]
We are left with proving that we can find an odd function $\phi$ and an even function $f$.
Indeed, if we have found a solution $\tilde P(t)$ with Taylor expansion given by $\hat P$ as above, then $P(t)\coloneqq \tilde P(|t|)$
 is also smooth at $t=0$, as the power series $\hat P$ has only even degree terms.
 Moreover $P(t)$ is again a solution to the system \eqref{GRSeq3}, as it is immediate to check. This concludes the proof.
\end{proof}

\section{Existence of complete steady generalized Ricci solitons}\label{SecComplete}
In this section we prove our main result, namely the existence of $\G$-invariant complete steady generalized Ricci solitons on $\mathbb R^3$.
In particular, we will prove Theorem \ref{IntroCOMP} through several steps.

\smallskip

Let $(\phi,f):[0,L)\to\R^2$ be a maximal solution to the system \eqref{GRSeq3} with $L\in \mathbb R^+\cup\{+\infty\}$,
and let $q\coloneqq \phi^{(3)}(0)$. Clearly, $\phi >0$ on $(0,L)$.

First of all, we consider the relation \eqref{pint1} with $k^2=2$
\[
\left(2e^{2f} + f'' + 2\frac{\phi'}{\phi}f' {-} (f')^2\right)' = 0,
\]
and the relation \eqref{pint2}
\[
\left(2\frac{1-(\phi')^2}{\phi^2} + 4\frac{\phi'}{\phi}f' {-} (f')^2+  e^{2f}\right)' = 0.
\]
Taking the limit as $t\to 0^+$ and using the equations \eqref{GRSeq3}, we get
\beq\label{int1}
2e^{2f} + f'' + 2\frac{\phi'}{\phi}f' {-} (f')^2 = 6q+5,
\eeq
and using the equation for $f''$ we get
\beq\label{int2}
2\frac{1-(\phi')^2}{\phi^2} + 4\frac{\phi'}{\phi}f' {-} (f')^2+  e^{2f} = 6q+5.
\eeq

\begin{proposition}\label{PropfNeg}
If  $6q+5\leq0$, then $f<0$ on $(0,L)$.
\end{proposition}
\begin{proof}
As $2q+1<0$, we have $f''(0)<0$, so that $f$ is negative on some interval $(0,\epsilon)\subset (0,L)$.
Suppose there is a point $T\in(0,L)$ with $f(T)=0$. Then there exists a point $T_1\in (0,L)$ where $f$ realizes its minimum in the interval $[0,T]$.
Then $f''(T_1)\geq 0$ and $f'(T_1)=0$. On the other hand, by \eqref{int1} we have
\[
f''(T_1)= 6q+5 - 2e^{2f(T_1)} < 6q+5\leq 0,
\]
a contradiction. Therefore $f$ is always negative.
\end{proof}

In order to prove the next result on the monotonicity of $f$, we introduce new variables
following \cite[p.~18]{CC}. We set
\[
x\coloneqq \phi',\quad y\coloneqq 2\phi'-f'\phi,\quad z\coloneqq e^{f}\phi.
\]
We also introduce a new parameter $r$ with
\[
r(t) = \int_{t_o}^t \frac 1{\phi(s)}ds.
\]
Note that $\phi'(0)=1$ implies that $\lim_{t\to 0}r= -\infty$.
Using the new parameter $r$ the equations \eqref{GRSeq3} read
\[
\frac{dx}{dr} = \frac{dx}{dt}\, \phi = \phi\phi'' = \phi\phi'f'+1-\phi'^2 -  e^{2f}\phi^2,
\]
hence
\[
\frac{dx}{dr} = x^2-xy +1 - z^2.
\]
Similarly,
\[
\frac{dy}{dr} = x(y-2x)- z^2,
\]
and
\[
\frac{dz}{dr} = \phi(e^f f'\phi + e^f\phi') = z(2x-y) + xz = (3x-y)z.
\]
Summing up, we have
\begin{equation}\label{Syst}
\left\{\begin{split}
\frac{dx}{dr} &= x^2-xy +1 - z^2,\\
\frac{dy}{dr} &= x(y-2x)- z^2,\\
\frac{dz}{dr}  &=(3x-y)z.\\
\end{split}\right.
\end{equation}

Accordingly, we find that \eqref{int2} reads now
\beq\label{intnew}
2x^2-y^2+z^2+2= (6q+5)\phi^2.
\eeq

\begin{proposition}\label{PropfDecr}
If $6q+5<0$, then $f$ is strictly decreasing.
\end{proposition}
\begin{proof}
We consider the function $u(r)\coloneqq y-2x$ on a maximal interval $(-\infty,M)$ for some $M\in \mathbb R\cup\{+\infty\}$.
By equations \eqref{Syst} we have
\[
\frac{du}{dr} =  x(y-2x)- z^2 - 2(x^2-xy +1 - z^2) = -4x^2 + 3xy + z^2 -2.
\]
Now $u=2\phi' - \phi f'- 2\phi' = -\phi f'$ and when $r\to-\infty$ (which corresponds to $t=0$) the function $f'$ is negative (as $f''(0)<0$ and $f(0)=0$),
so that $u$ is positive on some interval $(-\infty,r_o)$. Suppose that $u(r_o)=0$, namely $y=2x$ at $r=r_o$. Then
\[
\left.\frac{du}{dr}\right|_{r=r_o} = 2x^2 + z^2 -2|_{r=r_o}\leq 0.
\]
On the other hand, when $y=2x$, equation \eqref{intnew} reads  $-2x^2+z^2+2= (6q+5)\phi^2<0$, hence $2x^2 > z^2+2$ and therefore
\[
\left.\frac{du}{dr}\right|_{r=r_o} = 2x^2 + z^2 -2|_{r=r_o} > (z^2+2 + z^2 - 2) = 2z^2 >0,
\]
a contradiction. Therefore $u>0$ and $f'<0$ everywhere.
\end{proof}

\begin{proposition}\label{phi'minore}
If $6q+5<0$, then $\phi'<1$ on the interval $(0,L)$.
\end{proposition}
\begin{proof}
We know that $\phi'(0)=1$ and $\phi'<1$ on some right interval around $t=0$.
Let $t_1$ be a point with $\phi'(t_1)=1$ and $\phi'(t)<1$ on $(0,t_1)$, so that $\phi''(t_1)\geq 0$. Using the first equation in \eqref{GRSeq3},
we have at $t=t_1$
\[
\phi''(t_1) = (f'- \phi e^{2f})|_{t=t_1} <0,
\]
by Proposition \ref{PropfDecr}.
\end{proof}

The next result shows that the completeness of the soliton follows from the monotoniticy of the function $\phi$ for sufficiently negative $q$.
\begin{proposition}\label{lte}
If $6q+5<0$ and $\phi' >0$ on $(0,L)$, then $L=+\infty$ and the soliton is complete.
\end{proposition}

\begin{proof}
Suppose $L<+\infty$. As $\phi$ is increasing and $\phi'<1$, we see that the limit $\lim_{t\to L^-}\phi(t)$ exists and it is finite.
By \eqref{int2} we have that $4\frac{\phi'}{\phi}f'-(f')^2$ is bounded in a neighborhood of $L$. As $f'$ is always negative,
and $\frac{\phi'}{\phi}$ is positive $(0,L)$, we have that
\[
4\frac{\phi'}{\phi}f'-(f')^2 \leq - (f')^2.
\]
Now, if $f'$ were unbounded in a neighborhood of $L$, then we would have a sequence $t_n\to L$ with $f'(t_n)\to-\infty$
and therefore $(4\frac{\phi'}{\phi}f'-(f')^2)|_{t_n}\to-\infty$, a contradiction.
As $f'$ is negative and bounded from below, we see that the limit $\lim_{t\to L^-}f(t)$ exists and is finite.
Now the equations say that $\phi''$ and $f''$ are also bounded and therefore $\phi'$ and $f'$ have a finite limit when $t\to L$.
This implies that we can extend the solution beyond $L$ and this concludes the proof.
This implies that we can extend the solution beyond $L$ and thus that $L=+\infty$. 

We now show that the soliton is complete. Let $\{p_n\}$ be a Cauchy sequence in $\R^3$.
If $p_n\neq0$ we can write $p_n = (t_n,x_n)$, with $t_n\in\R$ and $x_n\in S^2$.
For $m,n\in\mathbb{N}$ we have $p_n \in \G\cdot \gamma(t_n)$ and $p_m \in \G\cdot \gamma(t_m)$ and we see that
\[
d(p_n,p_m) \geq d(\gamma(t_n),\gamma(t_m)) = |t_n-t_m|,
\]
since the normal geodesic $\gamma$ minimizes the distance between orbits.
This shows that $\{t_n\}\subset \R$ is a Cauchy sequence and thus $\lim_{n\to \infty} t_n = t_o\in\R$.
On the other hand, the sequence $\{x_n\}\subset S^2$ (sub)converges to some $x_o\in S^2$.
Therefore $\lim_{n\to \infty}p_n = (t_o,x_o)$, and the claim follows.

\end{proof}

The next result provides a fundamental estimate for the function $\phi e^f$, which is the product of the aperature $\phi$ and $\frac 16|H|$. This estimate will be essential in proving the long time existence.
\begin{proposition}\label{PropEst}
Suppose that
$\phi''\leq 0$ on an interval $[0,A)\subseteq [0,L)$. If $6q+5<0$, then on $[0,A)$ we have
\beq\label{est}\phi e^f\leq (2|6q+5|)^{-1/4}.\eeq
\end{proposition}
\begin{proof} We consider the conservation law \eqref{int1}
$$2e^{2f} + f'' + 2\frac{\phi'}{\phi}f' - (f')^2 = 6q+5 = - C^2,$$
with $C\coloneqq \sqrt{|6q+5|}$. We multiply it by $e^{-f}$ to obtain
$$2e^{f} + (f''- (f')^2)e^{-f} + 2\frac{\phi'}{\phi}f'e^{-f}  = -C^2e^{-f},$$
hence
$$(f'e^{-f})' +  2\frac{\phi'}{\phi}f'e^{-f} = -2e^{f} - C^2e^{-f}.$$
As the function $\mathbb R^+\ni s\mapsto 2s+\frac{C^2}s$ has a global minimum at $s= \frac{C}{\sqrt 2}$ with minimum value $2\sqrt 2 C$,
we obtain
$$(f'e^{-f})' +  2\frac{\phi'}{\phi}f'e^{-f} \leq -2\sqrt 2 C,$$
hence
$$(f'e^{-f})'\phi^2 +  2\phi'\phi f'e^{-f} \leq -2\sqrt 2 C\phi^2,$$
so that
$$(f'e^{-f}\phi^2)'\leq -2\sqrt 2 C\phi^2.$$
It then follows that
$$f'e^{-f}\phi^2\leq -2\sqrt 2 C \int_0^t\phi^2(u)du,$$
i.e.,
$$-f'e^{-f}\phi^2\geq 2\sqrt 2 C \int_0^t\phi^2(u)du,$$
whence
$$-f'e^{-f}\geq 2\sqrt 2 C \frac 1{\phi^2}\int_0^t\phi^2(u)du,$$
and therefore
$$e^{-f} - 1 \geq 2\sqrt 2 C \int_0^t\frac 1{\phi^2(u)}\left(\int_0^u\phi^2(s)ds\right)du.$$
Note that the function $\frac 1{\phi^2}\int_0^t\phi^2(u)du$ is smooth at $t=0$ as $\phi'(0)=1$.
We then obtain that
$$
\phi e^f\leq\frac{\phi}{1+ 2\sqrt 2 C \int_0^t\frac 1{\phi^2(u)}\left(\int_0^u\phi^2(s)ds\right)du}\leq \frac{t}{1
+ 2\sqrt 2 C \int_0^t\frac 1{\phi^2(u)}\left(\int_0^u\phi^2(s)ds\right)du}.
$$
We now use the concavity of $\phi$. We have
$$\int_0^u\phi(s)ds \leq \left(\int_o^u\phi^2(s)ds\right)^{1/2} \left(\int_0^u1\, ds\right)^{1/2} = \sqrt u \left(\int_o^u\phi^2(s)ds\right)^{1/2},$$
so that
$$\int_o^u\phi^2(s)ds \geq \frac 1u  \left(\int_o^u\phi(s)ds\right)^{2}.$$
By concavity we have
$$\int_o^u\phi(s)ds\geq \frac 12 u\phi(u),$$
hence
$$\int_o^u\phi^2(s)ds\geq \frac 1u  \frac{u^2}4\phi^2(u) = \frac 14 u\phi^2(u),$$
so that
$$\int_0^t\frac 1{\phi^2(u)}\left(\int_0^u\phi^2(s)ds\right)du \geq \int_0^t\frac 1{\phi^2(u)}\frac 14 u\phi^2(u)du = \frac 18t^2.$$
It then follows that
$$\phi e^f\leq \frac{t}{1+ \frac{\sqrt 2}{4} C t^2}\leq (\sqrt 2 C)^{-1/2}.$$
\end{proof}

We know that $\phi'(0)=1$ and we consider an interval $[0,B)\subseteq [0,L)$ where $\phi'$ is positive.

\begin{proposition}\label{conc}
The function $\phi$ is concave on $[0,B)$ if $q<-\frac 56 -\frac{25}{12} = -\frac{35}{12}$.
\end{proposition}
\begin{proof}
We know that $\phi''(0)=0$ and that it is negative in some right neighborhood of the origin.
Suppose there exists $t_o\in[0,B)$ where $\phi''(t_o)=0$ and $\phi''(t)<0$ on $(0,t_o)$.
Then $\phi^{(3)}(t_o) \geq 0$. Using the equations \eqref{GRSeq3}, \eqref{linkeqs}, and the fact that $\phi''(t_o)=0$, we have at $t=t_o$
\[
\begin{split}
\phi^{(3)}(t_o) 	&= -\frac {1-\phi'^2}{\phi^2}\phi' + \phi'f'' - 2f'e^{2f}\phi -e^{2f}\phi' \\
			&= -\frac {1-\phi'^2}{\phi^2}\phi' + \phi' e^{2f} - 2f'e^{2f}\phi -e^{2f}\phi' \\
			&= \frac{1}{\phi}\left(\phi'f'-\phi e^{2f}\right)\phi' - 2f'e^{2f}\phi  \\
			&= - \phi' e^{2f} + \frac{f'}{\phi}\left(\phi'^2 - 2 \phi^2e^{2f}\right).
\end{split}
\]
{

We will show that the last expression in the RHS is negative, obtaining a contradiction.
First, we consider the auxiliary function $\psi\coloneqq \phi' - \sqrt 2 \phi e^{f}$, which is positive at $t=0$. We claim that it is positive on
the interval $[0,t_o)$.
Indeed, suppose there exists $t_1\in [0,t_o)$ where $\psi(t_1)=0$ and $\psi(t) >0$ on $[0,t_1)$.
Then $\psi'(t_1)\leq 0$. On the other hand we have at $t=t_1$
\begin{eqnarray*}\psi'(t_1)&=&  \phi'' - \sqrt 2 \phi'e^{f}-\sqrt 2\phi f'e^f \\
&=& \frac{1-\phi'^2}{\phi} + \phi'f' - e^{2f}\phi -\sqrt 2 \phi'e^{f}-\sqrt 2\phi f'e^f\\
&=& \frac{1-2\phi^2e^{2f}}{\phi} + \sqrt 2\phi e^f f' - e^{2f}\phi - 2 \phi e^{2f}-\sqrt 2\phi f'e^f\\
&=& \frac{1-2\phi^2e^{2f}}{\phi}   - 3 \phi e^{2f}\\
&=& \frac 1{\phi}\left(1 - 5\phi^2e^{2f}\right).
\end{eqnarray*}
If $q< -\frac 56 -\frac{25}{12}$, then by Proposition \eqref{PropEst} we have that
$5\phi^2e^{2f} \leq 5(2|6q+5|)^{-1/2} < 1$, and therefore $\psi'(t_1) >0$, a contradiction.

Now, since $\psi>0$ on $[0,t_o)\subseteq [0,B)$ and $\phi'>0$ on $[0,B)$, we have $\psi\left(\phi' + \sqrt 2 \phi e^{f}\right) = \phi'^2-2\phi^2e^{2f}>0$
on $[0,t_o)$ and therefore $\frac{f'}{\phi}[\phi'^2 - 2 \phi^2e^{2f}]\leq0$ on $[0,t_o]$, as $f$ is decreasing. Thus
\[
\phi^{(3)}(t_o) = - \phi' e^{2f} + \frac{f'}{\phi}\left(\phi'^2 - 2 \phi^2e^{2f}\right) < 0,
\]
which contradicts $\phi^{(3)}(t_o) \geq 0$. The thesis then follows.

}
\end{proof}

We are now ready to prove the long time existence, which follows from Proposition \ref{lte} and the next result.
\begin{proposition}\label{phicresce}
The function $\phi$ has positive derivative on $[0,L)$ if $q< -\frac{35}{12}$.
\end{proposition}
\begin{proof} We consider the first point $t_o\in[0,L)$ where $\phi'(t_o) =0$. Then $\phi''(t_o)\leq 0$. Moreover, using the first equation in \eqref{GRSeq3}, 
we have at $t=t_o$
\[
\phi''(t_o)= \frac{1}{\phi} - e^{2f}\phi.
\]
By Proposition \ref{conc} $\phi$ is concave on $[0,t_o)$. Then, by Proposition \ref{PropEst}
we have that $\phi^2e^{2f}\leq (2|6q+5|)^{-1/2}<1$ whence $\phi''(t_o)>0$, a contradiction.
\end{proof}

Summing up, we have proved the existence of a one-parameter family of complete steady solitons given by the pairs $(\phi,f)$
solving the system \eqref{GRSeq3} with $q=\phi^{(3)}(0) < -\frac{35}{12}$.
The solitons $(g_\ell,H_\ell)$ in Theorem \ref{IntroCOMP} are then obtained as the solitons corresponding to solutions $(\phi_\ell,f_\ell)$
to the system \eqref{GRSeq3} with $\phi_\ell^{(3)}(0) = q_\ell = -\frac{35}{12}-e^{-\ell}$ and $X_\ell=\nabla f_\ell$.
The fact that these solitons are positively curved follows from propositions \ref{phi'minore},~\ref{conc},~\ref{phicresce},
together with the fact that the radial sectional curvature $K_{\mathrm{rad}}(t)$ and the sectional curvature $K_{\mathrm{tan}}(t)$ of the orbit
$\SO(3)\cdot\gamma_t$ are given by
\[
K_{\mathrm{rad}}(t)=-\frac{\phi''}{\phi},\qquad K_{\mathrm{tan}}(t)=\frac{1-\phi'^2}{\phi^2},
\]
respectively (see \cite[Remark 2]{Bryant}).
These identities also provide the following geometric interpretation of the parameter $q$,
concluding the proof of \ref{COMPiii}) in Theorem \ref{IntroCOMP}.
\begin{proposition}\label{Propqgeom}
The parameter $q=\phi^{(3)}(0) =  \tfrac12\left( f''(0)-1\right)$ is related to the radial sectional curvature and to the sectional curvature
of the orbit $\SO(3)\cdot\gamma_t$ as follows
\[
q = -\lim_{t\to0^+}K_{\mathrm{rad}}(t) = -\lim_{t\to0^+}K_{\mathrm{tan}}(t).
\]
\end{proposition}
\begin{proof}
Using de l'H\^opital rule, we obtain
\[
\begin{split}
-\lim_{t\to0^+}K_{\mathrm{rad}}(t)	&= \lim_{t\to0^+}\frac{\phi''}{\phi} = \lim_{t\to0^+}\frac{\phi^{(3)}}{\phi'} = \phi^{(3)}(0)=q,\\
-\lim_{t\to0^+}K_{\mathrm{tan}}(t)	&= \lim_{t\to0^+}\frac{\phi'^2-1}{\phi^2} =  \lim_{t\to0^+}\frac{\phi''}{\phi} = q.
\end{split}
\]
\end{proof}

\begin{remark}
We emphasize that both the radial and the tangential sectional curvature depend on two derivatives of the metric, and thus of $\phi$.
These quantities are related to $f''(t) = \tfrac12\mathcal{L}_X(\xi,\xi)$ through the soliton equation.
This links, in particular, the limit of the sectional curvatures at the origin with $q = \phi^{(3)}(0) = \tfrac12\left( f''(0)-1\right)$.
\end{remark}

We now prove \ref{COMPiv}) in Theorem \ref{IntroCOMP}.
\begin{proposition}
The metrics $\{g_\ell\}_{\ell\in\mathbb R}$ are mutually non-isometric.
\end{proposition}
\begin{proof}
First of all, we observe that the connected component of the isometry group of each metric $g_\ell$ is isomorphic to $\SO(3)$.
Indeed, let $\G_\ell\coloneqq \mathrm{Iso}^o(\R^3,g_\ell)$, which contains $\SO(3)$.
It is clear that $\G_\ell$ and its subgroup $\SO(3)$ share the same orbits, as otherwise $\G_\ell$ would have an open orbit
and therefore it would act transitively by the completeness of $g_\ell$.
In that case, the Ricci tensor of the positively curved homogeneous metric $g_\ell$ would have a positive lower bound,
implying the compactness of the manifold, a contradiction.
Let now $\varphi:(\R^3,g_\ell)\to (\R^3,g_{\ell'})$ be an isometry.
Then, $\varphi$ maps $\G_\ell$-orbits onto $\G_{\ell'}$-orbits and therefore, up to composing $\varphi$ with an element of $\G_\ell$
and up to a translation $t\mapsto t+c$, we can suppose that $\varphi(\gamma_t)=\gamma_{\pm t}$.
As $\varphi$ preserves the sectional curvatures, we see that $q_\ell=q_{\ell'}$ by Proposition \ref{Propqgeom}.
\end{proof}

We have already shown that $\phi$ is concave and satisfies $0<\phi'<1$,
and that $f$ is negative and strictly decreasing whenever $q<-\frac{35}{12}$.
We now complete the proof of \ref{COMPi}) and \ref{COMPii}) in Theorem \ref{IntroCOMP}
by showing the asymptotic behavior of these functions.
We start with the following lemma.
\begin{lemma}\label{f} There exist $T>0$ and constants $a,M\in\R,$ with $a<0$, such that
\[
f(t) \leq at + M, \qquad t\geq T.
\]
In particular, $\lim_{t\to+\infty} e^f = \lim_{t\to+\infty}\phi e^f = 0$.\end{lemma}
\begin{proof} We consider the function $\psi(t)\coloneqq -2\frac{f'(t)}{\phi(t)} + (f'(t))^2$ and observe that $\lim_{t\to 0^+}\psi(t)= -2(2q+1)$.
We choose a positive $\epsilon$ with $0<\epsilon<-2q-3$ (note that $q<-\frac 32$ and such a positive $\epsilon$ exists) and we fix a point $T>0$
so that $\psi(T)<-2(2q+1)+\epsilon$ where $0<\epsilon<-2q-3$.
We now put $a\coloneqq f'(T)$. As $f''(0)<0$ we can choose $T$ small enough so that $f''(T)<0$,
hence $f'(t)<f'(T)$ in a small right neighborhood of $T$.
We claim that $f'(t)< f'(T)$ for all $t>T$. Indeed suppose we have a point $T_1>T$ with $f'(T_1)=f'(T)$ while $f'(t)<f'(T)$ for all $t\in (T,T_1)$.
Then $f''(T_1)\geq 0$. On the other hand, by \eqref{int1} we have
\[
\begin{split}
f''(T_1)	&= 6q+5-2e^{2f(T_1)} - 2f'(T_1)\frac{\phi'(T_1)}{\phi(T_1)} + (f'(T_1))^2 \\
		&\leq 6q+5 - 2f'(T)\frac{1}{\phi(T)} + (f'(T))^2 \\
		&= 6q+5 + \psi(T)< 6q+5 -2(2q+1) +\epsilon < 2q+3-2q-3 = 0.
\end{split}
\]
where we have used the fact that $0<\phi'<1$ and $f'$ is negative.
This contradiction shows that $f'(t)<f'(T) =a<0$ for all $t>T$. Then $f(t)<at +M$ for some constant $M$ and all $t>T$.
The last claim follows from the previous estimate and the fact that $\phi\leq t$, for all $t\in\mathbb R^+$.
\end{proof}

\begin{proposition}
The functions $\phi$ and $f$ have the following asymptotic behavior for $t\to+\infty$
\[
\phi \sim \sqrt{\frac{2t}{|6q+5|}},\qquad  f \sim -\sqrt{|6q+5|}\,t.
\]
\end{proposition}

\begin{proof}
We claim that $\lim_{t\to+\infty} \phi = +\infty$.
We know that $\phi$ is increasing, so that the limit $\lim_{t\to+\infty}\phi(t)=M\in \mathbb R^+\cup\{+\infty\}$ exists.
Suppose $M<+\infty$. As $\phi$ is concave, the limit $\lim_{t\to +\infty}\phi'(t)$ exists finite and it must be $0$, otherwise $\phi$ would be unbounded.
If we look at the conservation law \eqref{int2}, we see that
\beq\label{f'}
f'= 2\frac{\phi'}{\phi}\pm \sqrt{4\frac{\phi'^2}{\phi^2}+e^{2f}+2\frac{1-\phi'^2}{\phi^2} -6q-5}.
\eeq
As $\lim_{t\to+\infty}f=-\infty$ by Lemma \ref{f}, we see that by \eqref{f'} the limit $\lim_{t\to+\infty}f'(t)$ exists and is finite.
This implies that $\lim_{t\to+\infty}\phi'f' = 0$.
Using now the equation for $\phi''$ together with Lemma \ref{f},  we see that
$\lim_{t\to+\infty}\phi'' = \frac 1{M}>0$, a contradiction. This proves that $\lim_{t\to+\infty}\phi=+\infty$.
By \eqref{f'}, we see that $\lim_{t\to+\infty} f'= -\sqrt{-6q-5}$.

We now put $\psi\coloneqq\phi\phi'$. From the system \eqref{GRSeq3}, we obtain
\[
\psi'=1+f'\psi - \phi^2 e^{2f}.
\]
As $\lim_{t\to+\infty}\phi e^{f}= 0$ by Lemma \ref{f},
and $\lim_{t\to+\infty}f'= -\a$, where $\a\coloneqq\sqrt{-6q-5}$, for every $\epsilon >0$ there exists $M>0$ so that for all $t\geq M$
\[
1-(\a+\epsilon)\psi -\epsilon \leq \psi'\leq 1 + (\epsilon-\a)\psi.
\]
By integration, there exist $c_1,c_2,c_3,c_4\in\mathbb R$ so that for all $t\geq M$
\[
2\frac{1-\epsilon}{\a+\epsilon} t + c_1e^{(\a+\epsilon)(M-t)}+c_2\leq \phi^2\leq \frac{2}{\a-\epsilon}t + c_3e^{(\a-\epsilon)(M-t)}+c_4.
\]
Hence
\[
2\frac{1-\epsilon}{\a+\epsilon} \leq \liminf_{t\to +\infty}\frac {\phi^2}{t} \leq \limsup_{t\to\infty}\frac{\phi^2}{t}\leq \frac {2}{\a-\epsilon},
\]
so that letting $\epsilon\to 0$, we see that $\lim_{t\to+\infty}\frac{\phi^2}t=\frac 2\a$.
\end{proof}

\smallskip

We are left with assertion \ref{COMPv}) in Theorem \ref{IntroCOMP}.
In order to prove it, we first observe that the expression
$H= \frac{h}{t^2}\vol_o = \sqrt 2 \frac{\phi^2 e^f}{t^2} \vol_o$ follows from \eqref{h} with $k=\sqrt 2$.
As $\frac{\phi^2}{t^2}\leq 1$, by Lemma \ref{f} we see that there exist constants $a<0$, $K>0$ such that
\[
\sqrt 2 \frac{\phi^2 e^f}{t^2}\leq K e^{at},
\]
for sufficiently large $t$.
From this, assertion \ref{COMPv}) follows, and the proof of Theorem \ref{IntroCOMP} is complete.

\bigskip

\noindent
{\bf Acknowledgements.}
The authors were supported by GNSAGA of INdAM
and by the project PRIN 2022AP8HZ9 "Real and Complex Manifolds: Geometry and Holomorphic Dynamics".
The authors would like to thank Jeffrey Streets, Mario Garcia-Fern\'andez and Ramiro Lafuente for their valuable remarks,
as well as Reto Buzano and Susanna Terracini for interesting conversations.
They also thank the referee for very useful comments, which helped to improve the article.


\begin{thebibliography}{99}

\bibitem{Bryant}
R.L.~Bryant.
\newblock Ricci flow solitons in dimension three with SO(3)-symmetries.
\newblock Available at
\href{https://services.math.duke.edu/~bryant/3DRotSymRicciSolitons.pdf}{https://services.math.duke.edu/~bryant/3DRotSymRicciSolitons.pdf}

\bibitem{CFMP}
C. G. Callan, D. Friedan, E. J. Martinec, M. J. Perry.
\newblock  Strings in background fields.
\newblock {\em Nuclear Phys.~B} {\bf 262} (1985), 593--609.

\bibitem {CC}
B. Chow et al.
\newblock The Ricci Flow: Techniques and Applications. Part I: Geometric Aspects.
{\em Math Surveys and Monographs}, vol.~{\bf 135}, A.M.S. 2007.

\bibitem{EW}
J.~Eschenburg, M.Y.~Wang.
\newblock The initial value problem for cohomogeneity one Einstein metrics.
\newblock {\em J. Geom. Anal.}, {\bf 10} (2000), 109-137.

\bibitem{Gar}
M.~Garcia-Fern\'andez.
\newblock Ricci flow, Killing spinors, and T-duality in generalized geometry.
\newblock {\em Adv. Math.} {\bf350} (2019), 1059--1108.

\bibitem{GFS}
M.~Garcia-Fern\'andez, J.~Streets.
\newblock Generalized Ricci Flow.
\newblock {\em AMS University Lecture Series} {\bf 76}, 2021.

\bibitem{FLS}
E.~Fusi, R.~Lafuente, J.~Stanfield.
\newblock The homogeneous Generalized Ricci flow.
\href{https://arxiv.org/abs/2404.15749}{arXiv:2404.15749} 

\bibitem{LW}
J.~Lauret, C.~Will.
\newblock Bismut Ricci flat generalized metrics on compact homogeneous spaces.
\newblock {\em Trans. Am. Math. Soc.}  {\bf376} (2023), 7495--7519. Corrigendum ibidem.

\bibitem{L}
K. H.~Lee.
\newblock The Stability of Generalized Ricci Solitons.
\newblock {\em J. Geom. Anal.} {\bf33} (2023), 273.


\bibitem{M}
B.~Malgrange.
\newblock Sur le points singuliers des \'equations diff\'erentielles.
\newblock {\em Enseignement Math.} {\bf 20} (1974), 147--176.

\bibitem{OSW}
T. Oliynyk, V. Suneeta, E. Woolgar.
\newblock A gradient flow for worldsheet nonlinear sigma models.
\newblock{\em Nuclear Phys.~B} {\bf 739} (2006), 441--458.

\bibitem{Oneill}
B.~O'Neill.
\newblock Semi-{R}iemannian geometry.
\newblock vol.~{\bf103} of {\em Pure and Applied Mathematics}.
\newblock Academic Press, New York 1983.

\bibitem{Par}
F. Paradiso.
\newblock Generalized Ricci flow on nilpotent Lie groups.
\newblock {\em Forum Math.} {\bf33} (2021), 997--1014.

\bibitem{PR1}
F.~Podest\`a, A.~Raf{}fero.
\newblock Bismut Ricci flat manifolds with symmetries.
\newblock {\em Proc. Roy. Soc. Edinburgh Sect. A} \textbf{153} (2023), 1371--1390.

\bibitem{PR2}
F.~Podest\`a, A.~Raf{}fero.
\newblock Infinite families of homogeneous Bismut Ricci flat manifolds.
\newblock {\em Commun. Contemp. Math.}  \textbf{26} (2024), 2250075.

\bibitem{RV}
A.~Raffero, L.~Vezzoni.
\newblock On the dynamical behaviour of the generalized Ricci flow.
\newblock {\em J.~Geom.~Anal.} \textbf{31} (2021), 10498--10509.

\bibitem{Str}
J.~Streets.
\newblock Regularity and expanding entropy for connection Ricci flow.
\newblock {\em J.~Geom.~Phys.} {\bf 58} (2008), 900--912.

\bibitem{Str1}
J. Streets.
\newblock Generalized geometry, T-duality, and renormalization group flow.
\newblock {\em J.~Geom.~Phys.} {\bf114} (2017), 506--522.

\bibitem{Str2}
J. Streets.
\newblock Classification of solitons for pluriclosed flow on complex surfaces.
\newblock {\em Math.~Ann.} {\bf375} (2019), 1555--1595.

\bibitem{StTi}
J.~Streets, G.~Tian.
\newblock A parabolic flow of pluriclosed metrics.
\newblock  {\em Int.~Math.~Res.~Not.}  {\bf16} (2010), 3101--3133.

\bibitem{StTi1}
J. Streets, G. Tian.
\newblock  Generalized K\"ahler geometry and the pluriclosed flow.
\newblock {\em Nuclear Phys. B} {\bf858}, 366--376, 2012.

\bibitem{StTi2}
J.~Streets, G.~Tian.
\newblock Regularity results for pluriclosed flow.
\newblock {\em Geom. Topol.} {\bf17} (2013), 2389--2429.

\bibitem{StUs}
J. Streets, Y. Ustinovskiy.
\newblock Classification of Generalized K\"ahler--Ricci Solitons on Complex Surfaces.
\newblock {\em Comm.~Pure Appl.~Math.} {\bf74} (2021), 1896--1914.

\bibitem{StUs2}
J. Streets, Y. Ustinovskiy.
\newblock The Gibbons--Hawking Ansatz in Generalized K\"ahler Geometry.
\newblock {\em Commun. Math. Phys.} {\bf 391} (2022), 707--778.


\end{thebibliography}
\end{document}